\documentclass{amsart}
\usepackage{amssymb}
\usepackage{amsmath, amscd}
\usepackage{amsthm}
\usepackage{mathrsfs}
\usepackage{perpage}
\usepackage{mathdots}
\usepackage[all,cmtip]{xy}
\usepackage{setspace}
\usepackage{amsbsy}
\usepackage{chngcntr}
\usepackage[foot]{amsaddr}
\usepackage{url}
\usepackage[T1]{fontenc}
\usepackage{verbatim}
\usepackage{mathrsfs}
\usepackage{ifthen}
\usepackage{blkarray}
\usepackage{multirow}
\usepackage{enumitem}
\usepackage{hyperref}
\usepackage{etoolbox}
\usepackage[left=2cm,head=13pt,
top=1.5cm,right=2cm, bottom=1.5cm]{geometry}

\setenumerate[0]{label=(\alph*)}

\title[The Griffiths bundle is group-theoretic]{The Griffiths bundle is generated by groups}

\AtEndDocument{\bigskip{\footnotesize%
\noindent (Wushi Goldring)
  \textsc{Department of Mathematics, Stockholm University, Stockholm SE-10691, Sweden} \par
 \noindent \texttt{wushijig@gmail.com} \par
 }}

\newtheorem{theorem}[subsubsection]{Theorem}
\newtheorem{lemma}[subsubsection]{Lemma}

\newtheorem{corollary}[subsubsection]{Corollary}

\newtheorem{question}[subsubsection]{Question}

\theoremstyle{remark}

\newtheorem{rmk}[subsubsection]{Remark}

\newcommand{\GZip}{\mathop{\text{$G$-{\tt Zip}}}\nolimits}
\newcommand{\QVHS}{\mathop{\text{$\QQ$-$\mathsf{VHS}$}}\nolimits}
\newcommand{\GGZip}{\mathop{\text{$\GG$-{\tt Zip}}}\nolimits}

\newcommand{\GLnZip}{\mathop{\text{$GL(n)$-{\tt Zip}}}\nolimits}

\DeclareMathOperator{\pr}{pr}

\DeclareMathOperator{\opp}{opp}

\DeclareMathOperator{\Grif}{Grif}
\DeclareMathOperator{\grif}{grif}
\DeclareMathOperator{\ogrif}{\overline{grif}}

\newskip\procskipamount
\newskip\interskipamount
\newskip\refskipamount
\newcommand{\procskip}{\vskip\procskipamount}
\newcommand{\interskip}{\vskip\interskipamount}
\newcommand{\refskip}{\vskip\refskipamount}

\newcommand{\procbreak}{\par
   \ifdim\lastskip<\procskipamount\removelastskip
   \penalty-100
   \procskip\fi
   \noindent\ignorespaces}

\newcommand{\titlebreak}{\par%
\ifdim\lastskip<\interskipamount\removelastskip%
\penalty10000%
\interskip\fi%
\noindent}%

\newcommand{\interbreak}{\par%
\ifdim\lastskip<\interskipamount\removelastskip%
\penalty-100%
\interskip\fi%
\noindent\ignorespaces}%

\newcommand{\refbreak}{\par%
\ifdim\lastskip<\refskipamount\removelastskip%
\penalty-100%
\refskip\fi%
\noindent\ignorespaces}%

\newcounter{listcounter}
\newcounter{deflistcounter}
\newcounter{equivcounter}

\newskip{\itemsepamount}
\newskip{\topsepamount}
\newenvironment{assertionlist}{%
  \begin{list}
    {\upshape (\arabic{listcounter})}
    {\setlength{\leftmargin}{18pt}
     \setlength{\rightmargin}{0pt}
     \setlength{\itemindent}{0pt}
     \setlength{\labelsep}{5pt}
     \setlength{\labelwidth}{13pt}
     \setlength{\listparindent}{\parindent}
     \setlength{\parsep}{0pt}
     \setlength{\itemsep}{\itemsepamount}
     \setlength{\topsep}{\topsepamount}
     \usecounter{listcounter}}}
  {\end{list}}

\newenvironment{definitionlist}{%
  \begin{list}
    {\upshape (\alph{deflistcounter})}
    {\setlength{\leftmargin}{18pt}
     \setlength{\rightmargin}{0pt}
     \setlength{\itemindent}{0pt}
     \setlength{\labelsep}{5pt}
     \setlength{\labelwidth}{13pt}
     \setlength{\listparindent}{\parindent}
     \setlength{\parsep}{0pt}
     \setlength{\itemsep}{\itemsepamount}
     \setlength{\topsep}{\topsepamount}
     \usecounter{deflistcounter}}}
  {\end{list}}

\newenvironment{equivlist}{%
  \begin{list}
    {\upshape (\roman{equivcounter})}
    {\setlength{\leftmargin}{18pt}
     \setlength{\rightmargin}{0pt}
     \setlength{\itemindent}{0pt}
     \setlength{\labelsep}{5pt}
     \setlength{\labelwidth}{13pt}
     \setlength{\listparindent}{\parindent}
     \setlength{\parsep}{0pt}
     \setlength{\itemsep}{\itemsepamount}
     \setlength{\topsep}{\topsepamount}
     \usecounter{equivcounter}}}
  {\end{list}}

% ...

\newcommand{\Xcal}{{\mathcal X}}

\newcommand{\CC}{\mathbf{C}}

\newcommand{\GG}{\mathbf{G}}

\newcommand{\LL}{\mathbf{L}}

\newcommand{\QQ}{\mathbf{Q}}
\newcommand{\RR}{\mathbf{R}}
\renewcommand{\SS}{\mathbf{S}}
\newcommand{\TT}{\mathbf{T}}

\newcommand{\XX}{\mathbf{X}}

\newcommand{\ZZ}{\mathbf{Z}}

\newcommand{\Cscr}{{\mathscr C}}

\newcommand{\Gscr}{{\mathscr G}}

\newcommand{\Vscr}{{\mathscr V}}

\newcommand{\cent}{{\rm Cent}}
\newcommand{\gm}{\GG_{\textnormal m}}
\newcommand{\gmFbar}{\GG_{\textnormal{m}, \Fbar}}

\newcommand{\Fbar}{\overline{F}}

\newcommand{\fp}{\mathbf F_p}

\newcommand{\rk}{\mbox{ rk }}

\DeclareMathOperator{\stab}{Stab}

\newcommand{\gal}{{\rm Gal}}

\newcommand{\galf}{\gal(\overline{F}/F)}

\newcommand{\res}{\mathsf{Res}}

\DeclareMathOperator{\ad}{ad}
\DeclareMathOperator{\Ad}{Ad}

\newcommand{\der}{{\rm der}}

\newcommand{\Hom}{{\rm Hom}}

\newcommand{\id}{{\rm Id}}

\renewcommand{\Vec}{\mathsf{Vec}}
\newcommand{\Rep}{\mathsf{Rep}}

\newcommand{\Par}{\mathsf{Par}}

\newcommand{\Th}{{\rm Th.}}

\newcommand{\Chap}{{\rm Chap.}}

\newcommand{\Prop}{{\rm Prop.}}

\newcommand{\no}{{\rm n\textsuperscript{o}}.}

\newcommand{\loccit}{{\em loc.\ cit. }}
\newcommand{\loccitn}{{\em loc.\ cit.}}
\newcommand{\cf}{{\em cf. }}
\newcommand{\ie}{i.e.,\ }

\newcommand{\etc}{etc.\ }
\newcommand{\esp}{esp.\ }

\newcommand{\fil}{{\rm Fil}}
\newcommand{\Gr}{{\rm Gr}}

\newcommand{\Sh}{{\rm Sh }}

\newcommand{\dR}{{\rm dR}}

\renewcommand{\Im}{{\rm Im}}

\DeclareMathOperator{\class}{Class}

\DeclareMathOperator{\card}{Card}

\DeclareMathOperator{\aut}{Aut}

\author{Wushi Goldring}
\date{\today}
\makeatletter
\let\c@equation=\c@subsubsection

\makeatother
\begin{document}
\pagestyle{plain}

\begin{abstract} First the Griffiths line bundle of a $\QQ$-VHS $\Vscr$ is generalized to a Griffiths character $\grif(\GG, \mu,r)$ associated to any triple $(\GG, \mu, r)$, where $\GG$ is a connected reductive group over an arbitrary field $F$, $\mu \in X_*(\GG)$ is a cocharacter (over $\Fbar$) and $r:\GG \to GL(V)$ is an $F$-representation; the classical bundle studied by Griffiths is recovered by taking $F=\QQ$, $\GG$ the Mumford-Tate group of $\Vscr$,  $r:\GG \to GL(V)$ the tautological representation afforded by a very general fiber and pulling back along the period map the line bundle associated to $\grif(\GG, \mu, r)$. The more general setting also gives rise to the Griffiths bundle in the analogous situation in characteristic $p$ given by a scheme mapping to a stack of $\GG$-Zips.
 
When $\GG$ is $F$-simple, we show that, up to positive multiples, the Griffiths character $\grif(\GG,\mu,r)$ (and thus also the Griffiths line bundle) is essentially independent of $r$ with central kernel, and up to some identifications is given explicitly by $-\mu$. As an application, we show that the Griffiths line bundle of a projective $\GGZip^{\mu}$-scheme is nef.
\end{abstract}
\maketitle
\tableofcontents
\pagebreak

\section{Introduction}
\label{sec intro}
We are motivated by the general problem of understanding which geometric objects are \emph{generated by groups}. We understand an object to be generated by groups if it is constructible and/or describable in terms of groups and associated data, such as subgroups, homogeneous spaces, representations, characters, cocharacters etc. In particular, in the case of reductive groups -- the focus of this paper -- we deem any geometric object which is describable in terms of root data of reductive groups as generated by groups.   

As a more precise example of the general problem, we begin this paper by stating some questions about the group-generation of   invariants of objects in a neutral Tannakian category. The paper is then concerned with showing that these questions have a particularly simple, explicit and positive answer when the invariant is the Griffiths line bundle of a variation of Hodge structure, or more generally the Griffiths character associated to a connected, reductive group $G$ over an arbitrary field, a cocharacter $\mu \in X_*(G)$ and a representation $r$ of $G$.
\subsection{Tannakian generation}
\label{sec-intro-tannakian}
There are many neutral Tannakian categories whose objects have been studied in algebraic geometry independently of Tannakian categories. A key example in this paper will be the category $\QVHS_{S}$ of variations of $\QQ$-Hodge structure over a smooth, projective $\CC$-scheme $S$.

Let $\Cscr$ be a  Tannakian category over a field $k$ which is neutralized by a fiber functor $\omega:\Cscr \to \Vec_k$. Let $G$ be the Tannaka group of $(\Cscr,\omega)$, \ie the affine $k$-group scheme which represents the functor of automorphisms $\underline{\aut}^{\otimes}(\omega)$ (see \cite[2.11]{Deligne-Milne-Tannakian}).
\begin{question}
\label{q-tannaka-1}
Given an invariant $i(X)$ associated to every object $X \in \Cscr$, is $i(X)$ generated by $G$ and additional group-theoretic data attached to $G$?
\end{question}
The prototypical type of additional group-theoretic data which we have in mind is a cocharacter $\mu \in X_*(G)$. If $G$ is reductive, then more generally any data deduced from a root datum of $G$ would qualify.

A somewhat more local variant of Question~\ref{q-tannaka-1} is the following: For every $X \in \Cscr$, let $G(X)$ be the Tannaka group of the Tannakian sub-category $\langle X \rangle ^{\otimes}$ generated by $X$. Then is $i(X)$ generated by $G(X)$ and group-theoretic data associated to $G(X)$? In this setting, one can even hope for more:
\begin{question}
\label{q-Tannakian-2} Assume some invariant $i$ is generated by $G(X)$ and some additional data associated to $G$. Is $i(X)$ essentially independent of $X$ (and dependent only on $G(X)$ and the additional data)?
\end{question}
A key component of Question~\ref{q-Tannakian-2} is of course to make precise the meaning of "essentially" in specific examples. The main result of this paper implies that Question~\ref{q-Tannakian-2} has a positive answer when $X$ is a $\QQ$-VHS, $i(X)$ is its Griffiths line bundle (\S\ref{sec-intro-grif}) and $G(X)$ is its Mumford-Tate group, provided the adjoint group $G(X)^{\ad}$ is $\QQ$-simple, see Theorem~\ref{th-main}. In this case, the additional group-theoretic data is the Hodge cocharacter $\mu \in X_*(G(X))$ and "essentially" means that the positive ray spanned by the Griffiths line bundle in the Picard group of the base is independent of $i(X)$ and dependent only on the pair $(G(X), \mu)$.
\subsection{The Griffiths bundle of a variation of Hodge structure}
\label{sec-intro-grif}
The Griffiths line bundle arose historically in Hodge theory, where it was used by Griffiths to study the algebraicity of the period map of a variation of Hodge structure \cite{Griffiths-IHES-period-integrals}. Suppose $S$ is a connected, smooth, finite-type $\CC$-scheme and $\Vscr$ is a (pure) polarized variation of Hodge structure on
$S$ with monodromy group $\Gamma$ and period domain $D$. Let $\fil^{\bullet} \Vscr$ be the (descending) Hodge filtration on $\Vscr$; for the sake of exposition suppose that
$\fil^{0}\Vscr=\Vscr$.
Griffiths (\loccitn, (7.13)) associated to $\Vscr$ the line bundle 
\begin{equation}
\label{eq-intro-def-grif}
\grif(\Vscr)=\det \bigoplus_{a \geq 1} \fil^a \Vscr=\det \bigoplus_{a \geq 1} (\Gr^a \Vscr)^{\oplus a}.
\end{equation}  We  call $\grif(\Vscr)$ the \underline{Griffiths line bundle}\footnote{In \loccitn, $\grif(\Vscr)$ was called the canonical bundle of $\Vscr$; in \cite{Green-Griffiths-Laza-Robles-compactification} and \cite{Bakker-Brunebarbe-Tsimerman-o-minimal-GAGA} it was called the Hodge (line) bundle.} of $\Vscr$.
Griffiths also associated to $\Vscr$ a period map \begin{equation}
\label{eq-intro-period-map}
    \Phi: S \to \Gamma \backslash D.
\end{equation}
By studying the positivity properties of the line bundle $\grif(\Vscr)$, Griffiths concluded that the image of the period map $\Phi(S)$ is projective algebraic when $\Gamma$ is discrete in $\aut(D)$ and the base $S$ is assumed projective
(\loccitn, (9.7); see also \cite[13.1.9]{carlson-muller-stach-peters-period-maps-domains}).\footnote{The condition "$\Gamma$ discrete in $\aut(D)$" is already necessary to ensure that $\Phi(S)$ is an analytic space.}
\begin{rmk}
\label{rmk-generalize-griffiths}
It is clear from the definition~\eqref{eq-intro-def-grif} that to define $\grif(\Vscr)$ requires much less than a polarized $\QQ$-VHS.
\begin{enumerate}
    \item The definition of $\grif(\Vscr)$ only depends on the associated graded $\Gr^{\bullet}\Vscr$.
    \item The transversality, polarization and $\QQ$-structure are not used: The same definition applies to any filtered or graded vector bundle on $S$. (By contrast the former data is crucial for defining the period map~\eqref{eq-intro-period-map}.)
\end{enumerate}
\end{rmk}
\begin{rmk}
\label{rmk-def-grif-sum-over-Z}
At least when the monodromy $\Gamma$ is discrete, it follows by pullback along the period map that the line bundle $\det\Vscr$ is trivial on $S$, see \S\ref{sec-Deligne-pairs-line-bundle}. Therefore one can replace the index set $a \geq 1$ in the sums~\eqref{eq-intro-def-grif} with $a \geq 0$ or $a \in\ZZ$; the resulting line bundle will be unchanged.
\end{rmk}
\begin{rmk}
\label{rmk-generalization-Griffiths-quasi-projective-noncompact}
For recent work generalizing that of Griffiths to the case that the base $S$ is only assumed quasi-projective, see the preprints \cite{Green-Griffiths-Laza-Robles-compactification}  and \cite{Bakker-Brunebarbe-Tsimerman-o-minimal-GAGA}. 
\footnote{It appears that the authors of \cite{Green-Griffiths-Laza-Robles-compactification} are working to fix an error in their preprint.}
\end{rmk}
\subsection{Summary of the paper}
\label{sec-summary}
Let $F$ be an arbitrary field. Consider triples $(\GG, \mu, r)$, where $\GG$ is a connected reductive $F$-group, $\mu \in X_*(\GG)$ is a cocharacter of $\GG_{\Fbar}$ and $r:\GG \to GL(V)$ is a morphism of $F$-groups.
In the vein of Remark~\ref{rmk-generalize-griffiths}, we explain in \S\ref{sec-griffiths-character} how to generalize the Griffiths line bundle to a character $\grif(\GG, \mu, r)$ of the Levi subgroup $\LL:=\cent_{\GG_{\Fbar}}(\mu)$ of $\GG_{\Fbar}$. In  \S\ref{sec-Deligne-pairs-line-bundle}, we describe how the Griffiths character gives rise to a Griffiths line bundle in two (a priori) different contexts: We recover the bundle $\grif(\Vscr)$ associated to a VHS via Deligne's theory of pairs $(G, \XX)$ and we obtain a Griffiths line bundle on stacks of $\GG$-Zips in characteristic $p>0$.

The main result is stated in \S\ref{sec-main-result}, see Theorem~\ref{th-main}. Roughly speaking, it states that $\grif(\GG, \mu, r)$ is, up to positive multiples and modulo the center, independent of $r$ and given explicitly by $-\mu$. To make this precise requires some technical assumptions and identifications concerning a root datum of $\GG$. For this purpose, the necessary structure theory associated to triples $(\GG, \mu, r)$ is given in \S\ref{sec-notation-structure-theory}. The sign change between $\mu$ and $\grif(\GG, \mu, r)$ reflects the change in positivity/curvature between a Mumford-Tate domain and its compact dual (Remark~\ref{rmk-sign-change}).

By combining our result with our forthcoming joint work with Y. Brunebarbe, J.-S. Koskivirta and B. Stroh \cite{Brunebarbe-Goldring-Koskivirta-Stroh-ampleness} on the positivity of automorphic bundles, we obtain the following application: Assume $X$ is a projective scheme in positive characteristic $p>0$ and $\zeta:X \to \GGZip^{\mu}$ is a morphism to the stack of $\GG$-Zips associated to $(\GG,\mu)$ by Pink-Wedhorn-Ziegler \cite{PinkWedhornZiegler-F-Zips-additional-structure,Pink-Wedhorn-Ziegler-zip-data}. As long as $p$ is not too small relative $\mu$ (orbitally $p$-close to be precise, \S\ref{sec-cond-characters}), then the Griffiths line bundle of $X$ is nef (Corollary~\ref{cor-nef}).

In \S\ref{sec-examples}, we give two examples of the main result: The first concerns the Hodge character and the Hodge line bundle. In the classical theory this amounts to the case that the VHS $\Vscr$ is polarized of weight one. Here we recover the results of our joint work with Koskivirta \cite{Goldring-Koskivirta-quasi-constant}. The second example provides explicit formulas for $\grif(\GG, \mu, r)$ when $r=\Ad$ is the adjoint representation, essentially in terms of the Coxeter number of the underlying root system.  

 The proof of Theorem~\ref{th-main} is given in \S\ref{sec-proof}. Some preliminary, general lemmas on roots and weights are given in \S\ref{sec-general-root-lemmas}; the proof proper occupies \S\ref{sec-proof-specific}. The key is to translate the main result into a statement about weight pairings with coroots (Lemma~\ref{lem-compute-grif-alpha}). The simply-laced case of Theorem~\ref{th-main} then results from a simple change of variables in the root pairing expression (Lemma~\ref{lem-sum-same-length}). The general case is reduced to the simply-laced one by the theory of root strings (Lemma~\ref{lem-sum-diff-length}). The application to nefness (Corollary~\ref{cor-nef}) is proved in \S\ref{sec-nefness}.
 
 When we sent P. Deligne a draft of this paper he quickly replied with a considerable simplification of the proof of the main result Theorem~\ref{th-main}\ref{item-ray-equality}. We are very grateful to Deligne for allowing us to include his simplification in Appendix~\ref{sec-deligne-simplification}.
 \section*{Acknowledgements}
 First, I want to thank my coauthors Y. Brunebarbe, J.-S. Koskivirta and B. Stroh for hours and hours of stimulating discussions on several topics related to this paper; in particular the idea for this work was born when Brunebarbe taught me about the Griffiths bundle and suggested to think of maps $\zeta:X \to \GZip^{\mu}$ as analogues of period maps. I am grateful to Koskivirta and J. Ayoub for comments on an earlier version of this work, which led to improvements in the statement of the main result.  I thank B. Moonen and T. Wedhorn for discussions about $G$-Zips and the link with classical Hodge theory. I am grateful to B. Klingler and J. Daniel for explaining to me the connection between the Griffiths bundle and Daniel's work on loop Hodge structures \cite{Daniel-loop-Hodge-structure-harmonic-bundles}; while we do not study this connection here, it would be interesting to understand the relationship between this paper and Daniel's work in the future.

This paper was completed during a visit to the University of Zurich. I thank the Institute of Mathematics for its hospitality and the opportunity to speak about this work.

We thank the referee for his/her work on the paper.

 Finally, it should be clear to the reader how much this paper owes to the works of Griffiths and Deligne. In addition to reshaping Hodge theory, we also thank them both for inspiring correspondence and discussions. In particular, I thank Griffiths for correspondence on positivity of Hodge bundles. I thank Deligne, first for sending me a two-line proof of the independence of $\grif(\GG,\mu,r)$ from $r$ when $\GG=GL(n)$, in response to my question about the dependence of $\grif(\GG,\mu,r)$ on $r$, and second for his simplification of the main result given in Appendix~\ref{sec-deligne-simplification}.
 \section{Notation and structure theory}
 \label{sec-notation-structure-theory}
\subsection{Notation}
\label{sec-notation}
Let $F$ be a field and fix an algebraic closure $\Fbar$ of $F$. A subscript $\Fbar, \QQ, \RR, \CC$ \etc will always denote base change to $\Fbar, \QQ, \RR, \CC$ respectively. Thus $\gmFbar$ denotes the multiplicative group scheme over $\Fbar$, and if $N$ is a $\ZZ$-module, then $N_{\QQ}:=N \otimes_{\ZZ}\QQ$ is the associated $\QQ$-vector space.

If $H$ is an algebraic $F$-group, then $X^*(H)=\Hom(H_{\Fbar},\gmFbar)$ (resp. $X_*(H)=\Hom(\gmFbar, H_{\Fbar})$) denotes the group of characters (resp. cocharacters) of $H_{\Fbar}$. Write $[\mu]$ for the $H(\Fbar)$-conjugacy class of a cocharacter $\mu \in X_*(H)$.

Similarly, $\Rep(H)$ denotes the category of representations of $H$ over $\Fbar$; write $\Rep_F(H)$ for the representations over $F$.

The unipotent radical of $H$ is denoted $R_u(H)$.

\subsection{Root data}
We follow \cite[\S2.1]{Goldring-Koskivirta-quasi-constant}, which in turn is based on \cite[Part II, \S1]{jantzen-representations}, \esp~1.18.
\subsubsection{Root datum of $\GG$}
\label{sec-root-datum-G}
Let $\TT$ be a maximal torus\footnote{The two possible interpretations of `maximal torus' here are equivalent \cite[Appendix A]{conrad-luminy-reductive-group-schemes}.} of $\GG$.
The root datum of $(\GG, \TT)$ is the quadruple
\begin{equation}
\label{eq-root-datum}
(X^*(\TT), \Phi; X_*(\TT), \Phi^{\vee}),
\end{equation}
together with the $\ZZ$-valued perfect pairing \begin{equation}
\label{eq-perfect-pairing}
    \langle, \rangle: X^*(\TT) \times X_*(\TT) \to \ZZ
\end{equation}
and the bijection $\Phi \to \Phi^{\vee}$, $\alpha \mapsto \alpha^{\vee}$,
where $\Phi:=\Phi(\GG, \TT)$ (resp. $\Phi^{\vee}:=\Phi^{\vee}(\GG, \TT)$) is the set of roots (resp. coroots) of $\TT_{\Fbar}$ in $\GG_{\Fbar}$.
\subsubsection{Weyl group}
\label{sec-Weyl-group} For every $\alpha \in \Phi$, let $s_{\alpha}$ be the corresponding root reflection. Let  $W:=W(\GG,\TT):= \langle s_{\alpha} | \alpha \in \Phi \rangle$ be the Weyl group of $\TT_{\Fbar}$ in $\GG_{\Fbar}$. 
Recall that $s_{\alpha} \leftrightarrow s_{\alpha^{\vee}}$ provides a canonical identification of $W$ with the dual Weyl group of the dual root datum generated by the $s_{\alpha^{\vee}}$.
\subsubsection{Weights}
\label{sec-weights}
If $V$ is an $F$-vector space and $r:\GG \to GL(V)$ is a morphism of $F$-groups, write $\Phi(V,\TT)$ for the set of $\TT_{\Fbar}$-weights in $V_{\Fbar}$ (not counting multiplicities). Given a weight $\chi \in \Phi(V, \TT)$, let $m_V(\chi)$ denote its multiplicity (the dimension of the corresponding weight space).

\subsubsection{Based root datum of $\GG$}
Let $\Delta \subset \Phi$ be a basis of simple roots. Then $\Delta^{\vee}:=\{\alpha^{\vee} \ | \ \alpha \in \Delta\}$ is the corresponding basis of simple coroots and
\begin{equation}
\label{eq-based-root-datum-G}
(X^*(\TT), \Delta, X_*(\TT), \Delta^{\vee})
\end{equation} is the based root datum of $(\GG, \TT, \Delta)$.

\subsection{Derived subgroup, adjoint quotient and simply-connected covering}
\label{sec-derived-subgroup}
Let $\GG^{\der}$ (resp. $\GG^{\ad}$, $\tilde \GG$) be the derived subgroup of $\GG$ (resp. its adjoint quotient, the simply-connected covering of  $\GG^{\der}$ in the sense of root data\footnote{One of several equivalent definitions is that the $\ZZ$-span of the coroots $\Phi^{\vee}$ is the whole cocharacter group.}). Let $s: \tilde \GG \to \GG$ be the natural quasi-section of the projection $\pr:\GG \to \GG^{\ad}$. The root datum~\eqref{eq-root-datum} and the based root datum~\eqref{eq-based-root-datum-G} naturally induce ones of $\GG^{\der}$, $\GG^{\ad}$ and $\tilde \GG$ as follows:
\subsubsection{Root datum of $G^{\der}$}
\label{sec-Gder-root-datum}
Let
\begin{equation}
\label{eq-def-char-Tder}
X_0(\TT):=\{\chi \in X^*(\TT) \ | \ \langle \chi, \alpha^{\vee} \rangle =0 \mbox{ for all } \alpha \in \Phi\}
\end{equation}
and
\begin{equation}
\label{eq-def-Tder}
\TT^{\der}:=\bigcap_{\chi \in X_0^*(\TT)} \ker \chi.
\end{equation}
Then $\TT^{\der}$ is a maximal torus in $\GG^{\der}$ with character group $X^*(\TT^{\der})=X^*(\TT)/X_0^*(\TT)$ and
\begin{equation}
\label{eq-root-datum-Gder}
(X^*(\TT^{\der}), \Phi; X_*(\TT^{\der}), \Phi^{\vee})
\end{equation}   is the root datum of $(\GG^{\der}, \TT^{\der})$, where the roots $\Phi$ are restricted to $\TT^{\der}$. Similarly, by restriction we identify $\Delta$ with a basis of simple roots for $(\GG^{\der}, \TT^{\der})$.
\subsubsection{Root data of $\GG^{\ad}$ and $\tilde{\GG}$}
\label{sec-Gad-tildG-root-data}
Let $\tilde \TT$ (resp. $\TT^{\ad}$) denote the preimage of $\TT^{\der}$ in $\tilde \GG$ (resp. the image of $\TT^{\der}$ in $\GG^{\ad}$). Then $\tilde \TT$ and $\TT^{\ad}$ are maximal tori in $\tilde \GG$ and $\GG^{\ad}$ respectively; the roots (resp. simple roots, coroots, simple coroots) of the three pairs $(\tilde \GG, \tilde \TT), (G^{\der}, \TT^{\der}),(\GG^{\ad}, \TT^{\ad})$ are identified via the central isogenies
\begin{equation}
\label{eq-central-isogenies}
(\tilde \GG, \tilde \TT) \to (\GG^{\der}, \TT^{\der}) \to(\GG^{\ad}, \TT^{\ad}).
\end{equation}
\subsubsection{Associated root system}
\label{sec-root-system}
The central isogenies~\eqref{eq-central-isogenies} induce canonical identifications
\begin{equation}
\label{eq-root-system}
X^*(\tilde \TT)_{\QQ}=X^*(\TT^{\der})_{\QQ}=X^*(\TT^{\ad})_{\QQ}
\textnormal{ and }
X_*(\tilde \TT)_{\QQ}=X_*(\TT^{\der})_{\QQ}=X_*(\TT^{\ad})_{\QQ}.
\end{equation}
In turn, the Weyl groups of the three pairs~\eqref{eq-central-isogenies} are all canonically identified with $W$. Choose a positive definite, symmetric, $W \rtimes \galf$-invariant bilinear form
\begin{equation}
\label{eq-def-bilinear-form}
    (,):X^*(\tilde \TT)_{\QQ} \times X^*(\tilde \TT)_{\QQ} \to \QQ.
\end{equation} When $\GG^{\ad}$ is $F$-simple, the form $(,)$ is the unique one up to positive scalar satisfying the properties above (one reduces to the well-known fact that there is a unique $W$-invariant, nondegenerate, symmetric form up to scaling when $\GG^{\ad}_{\Fbar}$ is simple; the latter follows from Schur's Lemma, because the natural representation $W \to GL(X^*(\tilde{\TT})_{\QQ})$ is then irreducible).

The triple $(X^*(\tilde\TT)_{\QQ}, \Phi, (,))$ is a root system associated  to $(\GG, \TT)$; its isomorphism type is independent of the choice of $(,)$.
\subsubsection{Fundamental weights}
\label{sec-fund-weights}
Given $\alpha \in \Delta$, write $\eta(\alpha)$ (resp. $\eta(\alpha^{\vee})$) for the corresponding fundamental weight (resp. fundamental coweight) of $(\tilde{\GG}, \tilde{\TT})$ in $X^*(\tilde \TT)_{\QQ}$  (resp. $X_*(\tilde \TT)_{\QQ}$) defined by
$\langle \eta(\alpha), \beta^{\vee} \rangle
=\delta_{\alpha \beta}=\langle \beta, \eta(\alpha^{\vee}) \rangle$ for all $\alpha, \beta \in \Delta$.
\subsection{Cocharacter data and associated subgroups}
\label{sec-cocharacter-Levi}
\subsubsection{Cocharacter data}
\label{sec-cocharacter-data}
Throughout much of this paper, we work with a pair $(\GG, [\mu])$, where $\mu \in X_*(\GG)$. Given such a pair, we choose a maximal torus $\TT$, a basis $\Delta \subset \Phi$ and a representative $\mu' \in [\mu]$ compatibly as follows: Choose $\TT$ over $F$ and $\mu'$ over $\Fbar$ such that $\Im(\mu') \subset \TT_{\Fbar}$; this is always possible because all maximal tori of $\GG_{\Fbar}$ are conjugate. In the presence of $\mu \in X_*(\GG)$, we always choose $\Delta$ so that $\mu$ is $\Delta$-dominant.
\subsubsection{Associated Levi subgroup}
\label{sec-associated-levi}
Given $\mu \in X_*(\GG)$, let $\LL$ be the Levi subgroup of $G_{\Fbar}$ given as the centralizer $\LL:=\cent_{\GG_{\Fbar}}(\mu)$. Let $\Delta_{\LL}:=\{\alpha \in \Delta \ | \ \langle \alpha, \mu \rangle =0 \}$. Then $\Delta_{\LL}$ is a basis of simple roots for $\TT_{\Fbar}$ in $\LL$.
\subsubsection{Parabolic subgroups and their flag varieties}
\label{sec-parabolics}
Let $I \subset \Delta$.
We define the standard parabolic subgroup of $\GG_{\Fbar}$ of type $I$ to be the subgroup generated by $\TT_{\Fbar}$ and the root groups $U_{\alpha}$ for $\alpha \in -\Delta \cup I$. In particular, the standard Borel subgroup is generated by $\TT_{\Fbar}$ and the root groups of negative roots.  Let $\Par_I$ be the flag variety of parabolics of $\GG_{\Fbar}$ of type $I$.
\subsubsection{Orbitally $p$-close cocharacters}
\label{sec-cond-characters}
Let $\mu \in X_*(\GG)$ and let $p$ be a prime number. Recall from \cite[\S N.5.3]{Goldring-Koskivirta-Strata-Hasse} (see also \cite[\S5.1]{Goldring-Koskivirta-quasi-constant})
that $\mu$ is called \underline{orbitally $p$-close} if, for every $\sigma \in W \rtimes \galf$ and every root $\alpha \in \Phi$ satisfying $\langle \alpha, \mu \rangle \neq 0$, one has \begin{equation}
\label{eq-def-orb-p-close}
    \left| \frac{\langle \sigma \alpha, \mu \rangle}{\langle  \alpha, \mu \rangle} \right| \leq p-1.
\end{equation}
A cocharacter which is orbitally $p$-close for all $p$ is called \underline{quasi-constant}. The condition "orbitally $p$-close" is a weakening of certain $p$-smallness conditions, while "quasi-constant" generalizes "minuscule". See \cite{Goldring-Koskivirta-quasi-constant} for more on quasi-constant (co)characters.
\section{The Griffiths character, the Griffiths bundle and the main result}
\label{sec-griffiths-character-main-result}
\subsection{The Griffiths character}
\label{sec-griffiths-character}
\subsubsection{Notation for real Hodge structures}
\label{sec-real-HS}
Let $\SS=\res_{\CC/\RR} \gm$. Recall that an $\RR$-Hodge structure consists of a pair $(V, h)$, where $V$ is a finite-dimensional $\RR$-vector space and $h:\SS \to GL(V)$ is a morphism of $\RR$-groups. Denote the (descending) Hodge filtration of $(V,h)$ on $V_{\CC}$ by $\fil^{\bullet} V_{\CC}$.  Let $\mu(z):=(h \otimes \CC)(z, 1)$ be the associated cocharacter of $GL(V_{\CC})$. By Deligne's convention, $\mu$ acts on the graded piece $\Gr^a V_{\CC}:=\fil^a V_{\CC}/\fil^{a+1} V_{\CC}$ by $z^{-a}$.
Let $\mu_{\max}$ (resp. $\mu_{\min}$) be the largest (resp. smallest) $\mu$-weight in $V_{\CC}$. Then
\begin{equation}
\label{eq-hodge-filt-min}
\fil^{-\mu_{\max}} V_{\CC}=V_{\CC}, \hspace{1cm} \fil^{1-\mu_{\min}} V_{\CC}=(0)
\end{equation}
and
$-\mu_{\max}$ (resp. $1-\mu_{\min}$) is characterized as the largest (resp. smallest)  integer satisfying~\eqref{eq-hodge-filt-min} (in other words, the Hodge filtration descends from $-\mu_{\max}$ to $-\mu_{\min}$).
\subsubsection{The Griffiths character for Deligne pairs}
\label{sec-Deligne-pairs}
Fix a pair $(G, \XX)$, where $G$ is a connected, reductive $\RR$-group  and $\XX:=\class_{\GG(\RR)}(h)$ is the $G(\RR)$-conjugacy class of a morphism of $\RR$-groups $h: \SS \to G$. The reinterpretation of much of Griffiths' work in terms of such pairs $(G,\XX)$ was introduced by Deligne in his Bourbaki talk \cite{Deligne-Travaux-Griffiths}.

Given $h \in \XX$, redefine $\mu(z):=(h \otimes \CC)(z,1) \in X_*(G)$ as the associated cocharacter of $G_{\CC}$.
Let $r:G \to GL(V)$ be a morphism of $\RR$-groups (later we will want to assume that $G,r$ both arise by base change from objects over $\QQ$).
Then $r \circ h$ is an $\RR$-Hodge structure. Define the \underline{Griffiths module} of $(G,h,r)$ by
\begin{equation}
\label{eq-def-Grif}
\Grif(G, h,r):=
\sum_{1-(r\circ \mu)_{\max}}^{-(r\circ \mu)_{\min}}
\fil^a V_{\CC}
\end{equation}
and the \underline{Griffiths character} of $(G,h,r)$ by
\begin{equation}
\label{eq-def-grif}
\grif(G, h, r):=\det \Grif(G, h, r).
\end{equation}

Let $L$ be the $\CC$-group $L:=\cent_{G_{\CC}}(\mu)$. Then $\Grif(G, h, r)$ is an $L$-module and $\grif(G,h,r) \in X^*(L)$ is a character of $L$. The isomorphism class of the $L$-module $\Grif(G, h,r)$ and the conjugacy class $[\grif(G, h, r)]$ do not depend on the choice of $h \in \XX$.
\subsubsection{Central kernel assumption}
\label{sec-central-kernel}
We shall always assume that $r:G \to GL(V)$ has central kernel; otherwise the component of the Griffiths character corresponding to some $\CC$-simple factor of $\tilde{\GG}_{\CC}$ will be trivial. For example $r=1$ trivial should clearly be avoided, for then $\grif(G,h,1)=0$ in $X^*(L)$.
\subsubsection{Generalization to arbitrary fields}
\label{sec-arb-fields}
Since the Hodge filtration on $V_{\CC}$ and the cocharacter $\mu$ uniquely determine each other, we can generalize the Griffiths module and character to the setting of arbitrary cocharacter data over arbitrary fields by working with $\mu$ instead of $h$.
Thus let $F$ be a field and fix an algebraic closure $\Fbar$ (no restriction is imposed on the characteristic of $F$). Let $\GG$ be a connected, reductive $F$-group and let $\mu \in X_*(\GG)$. For every $\Fbar$-vector space $V$ and every morphism $r:\GG_{\Fbar} \to GL(V)$ of $\Fbar$-groups, we have the cocharacter $r \circ \mu$ of $GL(V)$ and the corresponding descending filtration $\fil^{\bullet} V$ on $V$ given by
\begin{equation}
\label{eq-def-filt-from-cochar}
\fil^aV=\bigoplus_{a' \geq a}V_{-a'},
\end{equation}
where $V_{b}=\{v \in V \ | \ (r\circ \mu)(z)v=z^{b}v\}$ is the $b$-weight space of $r\circ \mu$ acting on $V$.

Given $r$ with central kernel (\S\ref{sec-central-kernel}), define the Griffiths module $\Grif(\GG, \mu, r)$  as in~\eqref{eq-def-Grif} and set  the Griffiths character to be its determinant: $\grif(\GG, \mu, r):=
\det \Grif(\GG, \mu, r)$.

As before, let $\LL:=\cent_{\Fbar}(\mu)$. Then
$\Grif(\GG,\mu,r)$ is an $\LL$-module and $\grif(\GG, \mu, r) \in X^*(\LL)$ for every triple $(\GG, \mu, r)$.
\subsection{Griffiths line bundles associated to Griffiths characters}
\label{sec-Deligne-pairs-line-bundle}
We explain how the Griffiths character gives rise to a line bundle by an associated bundle construction in two different settings: First we explain how to recover the Griffiths line bundle in the classical setting recalled in \S\ref{sec-intro-grif}. Then we describe the Griffiths line bundle in the context of $G$-Zips in positive characteristic.
\subsubsection{Associated sheaves}
\label{sec-associated-sheaves}
In both cases, one uses the following basic construction: Suppose $X$ is an $\Fbar$-scheme with an action of an algebraic $\Fbar$-group $H$. Then there is an exact tensor functor from $\Rep(H)$ to the category of vector bundles on the quotient stack $[H \backslash X]$ (\cite[N.4.1]{Goldring-Koskivirta-Strata-Hasse}, see also \cite[Part I, 5.8]{jantzen-representations} in the case $H$ acts freely on $X$). An equivalent variant is: The quotient stack $[H \backslash X]$ is equipped with a tautological $H$-torsor $I_H$; given $V \in \Rep(H)$, the pushout of $I_H$ via $V$ gives a $GL(V)$-torsor on $[H \backslash X]$, which is the torsor of bases of the vector bundle on $[H \backslash X]$ associated to $V$.
\subsubsection{Associated line bundle I: $F=\QQ$. The Borel embedding, d'apr\`es Deligne \cite[\S\S5.6-5.10]{Deligne-Travaux-Griffiths} } 
Return to the setting of \S\S\ref{sec-real-HS}-\ref{sec-Deligne-pairs}: 
$G$ is an $\RR$-group and $h:\SS \to G$, $r:G \to GL(V)$ are morphisms of $\RR$-groups; $\XX$ is the $G(\RR)$-conjugacy class of $h$.
Assume $r$ has central kernel and that  $r \circ h$ is a pure $\RR$-HS.

Then $P(h)=\stab_{G_{\CC}}(\fil^{\bullet} V_{\CC})$
is a parabolic subgroup of $G_{\CC}$ which is independent of $(V,r)$.
Let $I \subset \Delta$ be the type of $P(h)$. Then $I$ is independent of $h \in \XX$.
The compact dual of $\XX$ is the flag
variety $\check{\XX}:=\Par_I$.

The Borel embedding is the injection $\iota:\XX \to \check{\XX}(\CC)$ given by $h \mapsto P(h)$; it identifies $\XX$ with an open subset of $\check{\XX}(\CC)$, \cite[Lemma 5.8]{Deligne-Travaux-Griffiths}.

Applying~\S\ref{sec-associated-sheaves} to the tautological $P(h)$-torsor on $\Par_I=\check{\XX}$, combined with the functor $\Rep(L) \to \Rep(P)$ given by extending trivially on the unipotent radical $R_u(P)$, one gets a diagram of exact tensor functors:
\begin{equation}
\label{eq-associated-bundles}
\xymatrix{
\Rep(P) \ar[r]
&
\left(\parbox{2.5cm}{\centering $G_{\CC}$-equivariant\\ vector bundles \\ on $\check{\XX}$ } \right)
\ar^-{\iota^*}[d]
\\
\Rep(L) \ar[u] \ar@{-->}[r]
&
\left(\parbox{2.5cm}{\centering $G(\RR)$-equivariant\\ vector bundles \\ on $\XX$ } \right)
}
\end{equation}

Applying~\eqref{eq-associated-bundles} to the Griffiths character $\grif(G, h, r)=\grif(G, \mu,r)$ gives a $G(\RR)$-equivariant line bundle $\grif(\XX, r)$ on $\XX$.

The line bundle $\grif(\XX,r)$ is the line bundle associated to the tautological family of Hodge structures on $\XX$ (with fiber $r\circ h$ at $h \in \XX$) via \S\ref{sec-intro-grif}. However, this family of Hodge structures is rarely a VHS (\ie it rarely satisfies transversality): assuming that $\ad h$ is of weight $0$, the tautological family of HS over $X$ is a VHS if and only if $\ad h$ is of type $\{(1,-1),(0,0),(-1,1)\}$, \cite[\Prop~1.1.14]{Deligne-Shimura-varieties}.
Finally, one obtains the line bundle $\grif(\Vscr)$ of a $\Vscr \in \QVHS_S$ (\S\ref{sec-intro-grif}) by pullback along the period map~\eqref{eq-intro-period-map}: For every period domain $D$ (more generally for every Mumford-Tate domain $D$), there exists a Deligne pair $(G, \XX)$ such that $D=\XX^+$ (resp. $\aut(D)=G^{\ad}(\RR)^+$) is a classical topology connected component of $\XX$ (resp. $G(\RR)$).
\subsubsection{Associated line bundle II: $F=\fp$. $\GGZip^{\mu}$-schemes, d'apr\`es Pink-Wedhorn-Ziegler \cite{Pink-Wedhorn-Ziegler-zip-data,PinkWedhornZiegler-F-Zips-additional-structure}}
\label{sec-def-GZip-schemes}
Let $F=\fp$ and $\mu \in X_*(\GG)$.
Up to possibly conjugating $\mu$, we assume fixed a compatible choice of $\mu, \TT, \Delta$ as in \S\ref{sec-cocharacter-data}. Let $\GGZip^{\mu}$ be the associated stack of $\GG$-Zips of type $\mu$.
Let $P$ be the standard parabolic of type $\Delta_{\LL}$ (\S\ref{sec-parabolics}), $P^{\opp}$ its opposite relative $\LL$ and put $Q:=(P^{\opp})^{(p)}$.  Recall that a $\GG$-Zip of type $\mu$ on an $\Fbar$-scheme $S$ is a quadruple $(I,I_P,I_Q, \varphi)$, where $I$ is a $\GG$-torsor on $S$, $I_{P}$ (resp. $I_{Q}$) is a $P$-structure (resp. $Q$-structure) on $I$ and $\varphi:(I_P/R_u(P))^{(p)} \to I_Q/R_u(Q)$ is an isomorphism of $\LL^{(p)}$-torsors (\S\ref{sec-notation}).

Since part of the datum of a $\GG$-Zip of type $\mu$ is a $P$-torsor, every representation of $\LL$ yields a vector bundle on $\GGZip^{\mu}$ via \S\ref{sec-associated-sheaves}. We define the Griffiths line bundle $\grif(\GGZip^{\mu}, r)$ of $\GGZip^{\mu}$ to be the line bundle associated to the Griffiths character $\grif(\GG, \mu, r)$. If $X$ is an $\Fbar$-scheme and $\zeta:X \to \GGZip^{\mu}$ is a morphism, we define the Griffiths line bundle of $(X,\zeta)$ by pullback: $\grif(X, \zeta,r):=\zeta^*\grif(\GGZip^{\mu},r)$.
\subsubsection{$\GGZip^{\mu}$-schemes from de Rham cohomolgy, d'apr\`es Moonen-Wedhorn \cite{Moonen-Wedhorn-Discrete-Invariants} and Pink-Wedhorn-Ziegler \cite{PinkWedhornZiegler-F-Zips-additional-structure,Pink-Wedhorn-Ziegler-zip-data}}
\label{sec-GZip-schemes-from-de-Rham}
In order for \S\ref{sec-def-GZip-schemes} to be useful, one needs an interesting supply of morphisms $\zeta:X \to \GZip^{\mu}$.
In analogy with \S\ref{sec-intro-grif}, we recall how maps $\zeta:X \to \GGZip^{\mu}$ arise from de Rham cohomology in characteristic $p$, see also the introduction to \cite{Goldring-Koskivirta-global-sections-compositio}.
Suppose $\pi: Y \to X$ is a proper smooth morphism of schemes in characteristic $p$, that the Hodge-de Rham spectral sequence of $\pi$ degenerates at $E_1$ and that both the Hodge and de Rham cohomology sheaves of $\pi$ are locally free. Let $n=\rk H^i_{\dR}(Y/X)$ and consider the conjugacy class $[\mu]$ of cocharacters of $GL(n)$ whose $-a$-weight space has dimension $\rk R^{i-a} \pi_*\Omega_{Y/X}^a$.
Then $H^i_{\dR}(Y/X)$ is a $GL(n)$-Zip of type $\mu$; thus it determines a morphism $\zeta:X \to \GLnZip^{\mu}$.

The analogy between $\zeta$ and the period map $\Phi$ (\S\ref{sec-intro-grif}) was first suggested by Moonen-Wedhorn in the introduction to \cite{Moonen-Wedhorn-Discrete-Invariants}. We thank Y. Brunebarbe for suggesting to pursue this analogy further. It is an interesting open problem to understand what should be the right analogue of the Mumford-Tate group for $\zeta$. Still, if the Hodge filtration is compatible with certain tensors, then $\zeta$ will factor through a stack of $\GG$-Zips, where $\GG \subset GL(n)$ is the subgroup stabilizing those tensors. For example, as already observed in \cite{Moonen-Wedhorn-Discrete-Invariants}, when $Y/X$ is a family of polarized abelian schemes (resp. $K3$ surfaces) then $\zeta$ factors through a stack of $\GG$-Zips, where $\GG$ is a symplectic similitude group (resp.  $\GG=SO(21)$).

\subsection{Main result}
We continue to use the notation for root data introduced in \S\ref{sec-notation-structure-theory}; we always choose $\mu,\TT, \Delta$ compatibly as in \S\ref{sec-cocharacter-data} and the pairing $(,)$ on $X^*(\tilde \TT)_{\QQ}$ is always chosen $W \rtimes \galf$-invariant and positive definite (\S\ref{sec-root-system}).
\label{sec-main-result}
\subsubsection{The Griffiths ray}
Let \begin{equation}
    \ogrif(\GG, \mu, r):=\QQ_{>0} s^* \grif (\GG, \mu, r)
\end{equation}
be the positive ray generated by $s^* \grif (\GG, \mu, r)$ in $X^*(\tilde \TT)_{\QQ}$. We call $\ogrif(\GG, \mu, r)$ the \underline{Griffiths ray} of $(\GG, \mu, r)$.
More generally, if $v$ is a vector of a $\QQ$-vector space, write $\langle v \rangle:=\QQ_{>0}v$ for the ray which it generates.

Write $\mu^{\ad}:=\pr \circ \mu$ for the projection of $\mu$ onto $\GG_{\Fbar}^{\ad}$. Via the identifications~\eqref{eq-root-system}, one has $\mu^{\ad} \in X_*(\tilde \TT)_{\QQ}$.
\subsubsection{Weight pairing sums}
\label{sec-weight-pairing-sums}
Given a root $\gamma \in \Phi$ and a representation $r:\GG \to GL(V)$, recall \S\ref{sec-weights} and let
\begin{equation}
\label{eq-weight-pairing-sum}
S(\gamma^{\vee}, r)
:=
\sum_{\chi \in \Phi(V, \TT)}m_V(\chi)\langle \chi, \gamma^{\vee} \rangle ^2.
\end{equation}
Since $\langle s_{\alpha} \chi, \alpha^{\vee} \rangle=-\langle \chi, \alpha^{\vee}\rangle$, each summand in~\eqref{eq-weight-pairing-sum} is invariant under $\chi \mapsto s_{\alpha} \chi$.

\begin{theorem}
\label{th-main}
Let $\GG$ be a connected, reductive $F$-group. Assume $\GG^{\ad}$ is $F$-simple, $\mu \in X_*(\GG)$ and $r \in \Rep_F(\GG)$ is a representation of $\GG$ over $F$ with central kernel. Then:
\begin{enumerate}
\item
\label{item-grif-fund-weights}
For all $\alpha \in \Delta$, one has
\begin{equation}
\label{eq-grif-fund-weights}
\langle \grif(\GG, \mu, r), \alpha^{\vee} \rangle
=
-\frac{1}{2}\langle \alpha, \mu \rangle S(\alpha^{\vee},r).
\end{equation}
\item
\label{item-indep-alpha}
The value $(\alpha, \alpha)S(\alpha^{\vee},r)$ is independent of $\alpha \in \Phi$.
\item
\label{item-griffiths-explicit} Under the identification $X^*(\tilde{\TT})_{\QQ} \cong X_*(\tilde{\TT})_{\QQ}$ afforded by $(,)$, for every $\alpha \in \Phi$ one has
\begin{equation}
\label{eq-explicit-griffiths}
s^*\grif(\GG, \mu, r)
=
-\frac{(\alpha, \alpha)}{4}S(\alpha^{\vee},r) \mu^{\ad} \hspace{.5cm} \textnormal{ in } X^*(\tilde{\TT})_{\QQ}.
\end{equation}
\item
\label{item-ray-equality}
In particular,
\begin{equation}
\ogrif(\GG, \mu, r)
=
-\langle \mu^{\ad} \rangle \hspace{.5cm}\textnormal{ in } X^*(\tilde \TT)_{\QQ}.
\end{equation}
\end{enumerate}
\end{theorem}
\begin{rmk} \
\begin{enumerate}
    \item
Both of the terms $(\alpha, \alpha)$ and $\mu^{\ad}$, when viewed in $X^*(\tilde{\TT})_{\QQ}$, depend on $(,)$. Further, the dependence among the two is inverse proportional, so the right-hand side of~\eqref{eq-explicit-griffiths} is independent of $(,)$
\item For fixed $r$, the value $(\alpha, \alpha) S(\alpha^{\vee},r)$ depends on $\Phi \subset X^*(\tilde \TT)_{\QQ}$, not just on the isomorphism class of the root system $(X^*(\tilde{\TT})_{\QQ}, \Phi, (,))$.
\end{enumerate}
\end{rmk}

\begin{rmk}
\label{rmk-sign-change}
The sign difference between the cocharacter $\mu$ and the Griffiths character $\grif(\GG, \mu, r)$ is a reflection of the change in curvature/positivity between a period domain $\XX$ (or more generally a Mumford-Tate domain, or a Griffiths-Schmid manifold) and its compact dual $\check{\XX}$. For example, the Hodge line bundle is ample on a Hodge type Shimura variety (\S\ref{sec-hodge-line-bundle}), but over $\CC$ it arises via the construction~\ref{eq-associated-bundles} from a line bundle on $\Check{\XX}$ which is anti-ample on $\Check{\XX}$.
\end{rmk}
Under the same hypotheses, two immediate corollaries of Theorem~\ref{th-main} are:
\begin{corollary}[Independence]
\label{cor-independence} \
\begin{enumerate}
    \item Given a cocharacter datum $(\GG, \mu)$, the Griffiths ray $\ogrif(\GG, \mu,r)$ is independent of $r$ (always assumed with central kernel).
\item
\label{item-cor-proportionality-scalar}
Given $r \in \Rep_F(\GG)$ with central kernel, the positive scalar $c \in \QQ>0$ such that $s^*\grif(\GG,\mu,r)=-c\mu^{\ad}$ is independent of $\mu$.
\end{enumerate}

\end{corollary}
\begin{corollary}
\label{cor-main-simply-laced}
In addition to the hypotheses of \textnormal{Theorem~\ref{th-main}}, assume that all roots $\alpha$ satisfying $\langle \alpha, \mu \rangle \neq 0$ have the same length\footnote{In particular this holds if $\GG$ is simply-laced.}. Then, without reference to $(,)$, one has
\begin{equation}
\label{eq-ogrif-simply-laced}
    \ogrif(\GG, \mu, r)=-\langle \sum_{\alpha \in \Delta} \langle \alpha, \mu \rangle \eta(\alpha) \rangle.
\end{equation}
\end{corollary}

\begin{rmk} The assumption that $\GG$ is $F$-simple and the need to consider associated rays (\ie to allow positive scalar multiples) are both already essential in the setting of the Hodge line bundle, see \cite[\S4.5]{Goldring-Koskivirta-quasi-constant} and \cite[\S2.1.6, Footnote 7]{Goldring-Koskivirta-Strata-Hasse} for respective examples.

\end{rmk}
As an application of our joint work with Brunebarbe, Koskivirta and Stroh on positivity of automorphic bundles \cite{Brunebarbe-Goldring-Koskivirta-Stroh-ampleness}, one obtains the nefness of the Griffiths bundle on a proper $\GGZip^{\mu}$-scheme.
\begin{corollary}[Nefness]
\label{cor-nef} Assume $F$ has characteristic $p$, that $X$ is a proper $\Fbar$-scheme of finite type and that $\zeta:X \to \GGZip^{\mu}$ is a morphism. If $\grif(\GG, \mu, r)$ is orbitally $p$-close \textnormal{(\S\ref{sec-cond-characters})}, then the pullback of the Griffiths line bundle to $X$ is a nef line bundle on $X$.
\end{corollary}
\begin{rmk}
\label{rmk-compare-BGKS}
We emphasize that \cite{Brunebarbe-Goldring-Koskivirta-Stroh-ampleness} contains stronger positivity results and that our sole contribution here is to show that $\grif(\GG, \mu, r)$ is $\Delta \setminus \Delta_{\LL}$-negative, see Corollary~\ref{cor-anti-ample}.
\end{rmk}
\begin{rmk}
\label{rmk-nef-open-base}
Unlike the property "ampleness", the property "nef" is not always open on the base, \cf \cite{Moriwaki-criterion-openness-family-nef}. So we do not know if one can reprove Griffiths' result that the Griffiths line bundle of a polarized $\QQ$-VHS over a projective base is nef via Corollary~\ref{cor-nef}.
\end{rmk}
\subsection{Examples}
\label{sec-examples}

\subsubsection{The Hodge character and line bundle I: $F=\QQ$}
\label{sec-hodge-line-bundle}
Consider the special case of triples $(\GG, \mu,r)$  where
\begin{enumerate}
    \item $F=\QQ$ (\ie $\GG$ is a $\QQ$-group),
    \item $\mu(z)=(h \otimes \CC)(z,1)$ for some $h:\SS \to \GG_{\RR}$ and $(\GG, \XX)$ is a Shimura datum, where $\XX=\class_{\GG(\RR)}(h)$,
    \item $r\circ h$ is of type $\{(0,-1),(-1,0)\}$.  \end{enumerate}
Then the Shimura datum $(\GG, \XX)$ is of Hodge type and $r$ is a symplectic embedding \cite[Lemma 1.3.3]{Deligne-Shimura-varieties}.  In this case, the Griffiths character $\grif(\GG, \mu, r)$ is the \underline{Hodge character} giving rise to the Hodge line bundle of the Shimura variety $\Sh(\GG, \XX)$. Since $\mu$ is minuscule in this example, the condition of Corollary~\ref{cor-main-simply-laced} holds.  Then Corollary~\ref{cor-main-simply-laced} recovers Theorem~1.4.4 and Corollary~1.4.5 of our joint work with Koskivirta \cite{Goldring-Koskivirta-quasi-constant}, which state that the Hodge character is quasi-constant (\S\ref{sec-cond-characters})  and that the Hodge ray it determines is independent of $r$ and given by~\eqref{eq-ogrif-simply-laced}.

For applications of these results to the "tautological" ring of Hodge-type Shimura varieties and the cycle classes of Ekedahl-Oort strata, see the recent preprint of Wedhorn-Ziegler  \cite{Wedhorn-Ziegler-tautological}.
\subsubsection{The Hodge character and line bundle II: $\GG$-Zips}
\label{sec-hodge-line-bundle-GZip}
Let $F=\fp$. Let $(V,\psi)$ be a symplectic space over $\fp$ of dimension $g$ and $GSp(V,\psi)$ the corresponding symplectic similitude group. Let $\mu_g$ be a non-central, minuscule cocharacter of $GSp(V,\psi)$.  The Hodge character is defined for any symplectic embedding of cocharacter data
\begin{equation}
\label{eq-symplectic-embedding-mod-p}
(\GG, \mu) \hookrightarrow (GSp(2g), \mu_g)
\end{equation}
and the associated line bundle is the Hodge line bundle of $\GGZip^{\mu}$, \cite[\S1.3]{Goldring-Koskivirta-Strata-Hasse}. Theorem~\ref{th-main} extends the results of \cite{Goldring-Koskivirta-quasi-constant}, recalled in \S\ref{sec-hodge-line-bundle}, to symplectic embeddings~\ref{eq-symplectic-embedding-mod-p} which need not arise from an embedding of Shimura data. In particular, the Hodge character is quasi-constant even if it does not arise from a Shimura variety setting by reduction mod $p$.
\subsubsection{The adjoint representation via the Coxeter number}
\label{sec-example-adjoint}
When $r=\Ad$ is the adjoint representation, the sums $S(\alpha^{\vee},r)$ (\S\ref{sec-weight-pairing-sums}) are sums of squares of root pairings, and their values are computed explicitly using the computations by Bourbaki of the "canonical bilinear form" and $\gamma$-invariant of a root system (\cite[\Chap~6,~\S1.12]{bourbaki-lie-4-6} and exercise 5 of \Chap~6, \S1 in \loccitn; see also Remark~\ref{rmk-app-canonical-bilinear-form}). In this way, given the root system $(X^*(\tilde \TT)_{\QQ}, \Phi, (,))$ associated to $(\GG, \TT)$ (\S\ref{sec-root-system}), one obtains the proportionality constant $c$ of Corollary~\ref{cor-independence}\ref{item-cor-proportionality-scalar} for $r=\Ad$.

When $\Delta$ is irreducible and simply-laced, one has $S(\alpha^{\vee}, \Ad)=4h$ for all $\alpha \in \Phi$, where $h$ is the Coxeter number of $\Delta$. When $\Delta$ is irreducible and multi-laced one has two invariants of the root system:
$S(\alpha^{\vee}, r)$ for $\alpha$ short and long respectively.
These are recorded in Table~\ref{table-root-pairing-sums}, together with the Coxeter number and $\gamma(\Delta)$;
the latter are taken from the Planches in \loccit When $\Delta$ is simply-laced, one has $h^2=\gamma(\Delta)$.

\begin{table}[ht] \label{table-root-pairing-sums}
        \caption{The root pairing sums $S(\alpha^{\vee}, \Ad)$}

{\renewcommand{\arraystretch}{1.25}
\begin{tabular}{|c|c|c|c|c|c|}
\hline
 Type of $\Delta$
 &
\begin{tabular}{l}
Coxeter  \\
 number $h$
 \end{tabular}
&
 $\gamma(\Delta)$
&
\begin{tabular}{c}
$S(\alpha^{\vee}, \Ad)$ \\
$\Delta$ simply-laced
\end{tabular}
&
\begin{tabular}{c}
$S(\alpha^{\vee}, \Ad)$ \\
$\alpha^{\vee}$ short
\end{tabular}
&
\begin{tabular}{c}
$S(\alpha^{\vee}, \Ad)$ \\
$\alpha^{\vee}$ long
\end{tabular}
\\ \hline
 $A_{n-1}$ & $n$ & $n^2$ & $4n$ & &
\\ \hline
$B_n$ & $2n$ & $(n+1)(4n-2) $ & & $4(2n-1)$ & $8(2n-1)$

\\  \hline
$C_n$ &  $2n$ & $(n+1)(4n-2) $ &  & $4(n+1)$ & $8(n+1)$
\\ \hline
 $D_n$ &  $2(n-1)$ & $4(n-1)^2$ & $8(n-1)$ & &
\\ \hline
$G_2$ &  $6$ & $48$ & & $16$ & $48$
\\ \hline
$F_4$ &   $12$ & $162$ & & $36$ & $72$
\\ \hline
 $E_6$  & $12$ & $144$ &  $48$ & &
\\ \hline
 $E_7$ &  $18$ & $324$ & $72$ & &
\\ \hline
 $E_8$ &  $30$ & $900$ & 120 & &
\\ \hline
\end{tabular}}
\end{table}

\section{Proof of the main result and its application}
\label{sec-proof}
\subsection{General lemmas on roots and weights}
\label{sec-general-root-lemmas}
Let $k$ be an algebraically closed field.
\begin{lemma}
\label{lem-torus-rep}
Let $T$ be a $k$-torus. Assume  $$ \varphi: T \to GL(V)$$ is a finite-dimensional representation with finite kernel. Then the $\QQ$-span of the $T$-weights of $\varphi$ is all of $X^*(T)_{\QQ}$.
\end{lemma}

\begin{proof}
Let $T':=T/ \ker \varphi$. Then $T'$ is a torus,
and the induced map
$\overline{\varphi}:T' \to GL(V)$ is faithful. Moreover,  $X^*(T)_{\QQ}=X^*(T')_{\QQ}$ since $\ker \varphi$ is finite. Thus we reduce to the case that $\varphi$ is faithful.

Since $T$ is a torus, the category $\Rep_k(T)$ is semisimple and Tannakian, neutralized by the forgetful functor
$\Rep_k(T) \to \Vec_k$. Since $\varphi$ is faithful, every
$\chi \in X^*(T)$ is a factor of some $T^{m,n}(V):=V^{\otimes m} \otimes (V^{\vee})^{\otimes n}$.
In other words, $\chi$ is a  $T$-weight of some $T^{m,n}(V)$.
The $T$-weights of $T^{m,n}(V)$ are $\ZZ$-linear combinations of the $T$-weights of $V$.
So $\chi$ lies in the $\QQ$-span of the $T$-weights of $V$.
\end{proof}
Recall the notation $\Phi(V,T)$ for the $T$-weights of $V$ (\S\ref{sec-weights}).
\begin{lemma}
\label{lem-weights-not-all-root-hyperplane}
Assume that $r:G \to GL(V)$ is a representation with central kernel. Then, for every simple root $\alpha \in \Delta$, there exists a weight $\chi \in \Phi(V, T)$ such that $\langle \chi, \alpha^{\vee} \rangle \neq 0$.
\end{lemma}
\begin{proof}
Let $\iota: T^{\der} \to T$ be the inclusion. If $\chi$ is a $T$-weight of $V$, then $\langle \chi, \alpha^{\vee} \rangle=\langle \iota^* \chi, \alpha^{\vee} \rangle$, where we identify the coroots of $T^{\der}$ in $G^{\der}$ with those of $T$ in $G$ as in \S\ref{sec-Gder-root-datum}. Thus we reduce to the case that $G=G^{\der}$ is semisimple. Then $r$ has finite kernel, hence so does its restriction to $T \subset G$. Now the result follows from Lemma~\ref{lem-torus-rep}, for otherwise the $\QQ$-span of the weights would lie in a root hyperplane.
\end{proof}
\begin{corollary} 
\label{cor-weights-not-all-root-hyperplane}
Assume that $r:G \to GL(V)$ is a representation with central kernel and that $\mu \in X_*(T)$ is not central in $G$. Then there exists a weight $\chi \in \Phi(V,T)$ such that $\langle \chi, \mu \rangle \neq 0$.
\end{corollary}
\begin{proof}
Since $\mu$ is not central, there exists $\alpha \in \Delta$ such that $\langle \alpha, \mu \rangle \neq 0$. By Lemma~\ref{lem-weights-not-all-root-hyperplane} applied to $\alpha$, there exists $\chi \in \Phi(V,T)$ such that $\langle \chi, \alpha^{\vee} \rangle \neq 0$.

Since $\langle s_{\alpha}\chi, \mu \rangle =\langle \chi, \mu \rangle -\langle \chi, \alpha^{\vee} \rangle \langle \alpha,\mu \rangle$, we conclude that either $\langle \chi, \mu \rangle \neq 0$ or $\langle s_{\alpha}\chi, \mu \rangle \neq 0$. Since the set of weights $\Phi(V,T)$ is $W$-stable, the corollary is proved.
\end{proof}

Return to the setting that $F$ is an arbitrary field with algebraic closure $\Fbar$.
\begin{lemma}
\label{lem-weyl-gal-same-length}
Assume $\GG$ is an adjoint, simple $F$-group. Then two roots $\alpha, \beta \in \Phi$ have the same length if and only if they are conjugate under $W \rtimes \galf$.
\end{lemma}
\begin{proof}
The action of $W \rtimes \galf$ preserves length (recall from \S\ref{sec-root-system} that the bilinear form $(,)$ is chosen $W \rtimes \galf$-invariant).

Conversely, assume $\alpha, \beta \in \Phi$ are two roots of equal length. Since $\GG$ is simple over $F$, the $\Fbar$-simple factors of $\GG_{\Fbar}$ are permuted transitively by $\galf$.  Thus there exist $\Fbar$-simple factors $G_i$ and $G_j$ of $\GG_{\Fbar}$, with respective maximal tori $T_i$, $T_j$ contained in $\TT_{\Fbar}$, such that the root systems $(X^*(T_i), \Phi_i, (,)),(X^*(T_j), \Phi_j, (,))$ are naturally irreducible components of the root system $(X^*(\TT)_{\QQ}, \Phi, (,))$ and $\alpha \in \Phi_i$, $\beta \in \Phi_j$. Let $\sigma_0 \in \galf$ map $\Phi_i$ to $\Phi_j$. In a reduced and irreducible root system, two roots are conjugate under the Weyl group if and only if they have the same length; apply this to the system $\Phi_j$ and the roots $\sigma_0 \alpha, \beta$.
\end{proof}

\subsection{Root-theoretic analysis of the Griffiths module}
\label{sec-proof-specific}
Consider the setting of Theorem~\ref{th-main}: Let $\GG$ be a connected, reductive $F$-group, $V$  an $F$-vector space and $r:\GG \to GL(V)$ a morphism $F$-groups with central kernel.
Let $\mu \in X_*(\GG)$.

The expression of $\mu^{\ad}$ in the basis of fundamental coweights (\S\ref{sec-fund-weights}) in $X_*(\TT^{\ad})_{\QQ}$ is $$\mu^{\ad}=\sum_{\alpha \in \Delta} \langle \alpha, \mu \rangle \eta(\alpha^{\vee}).$$
Since $\alpha^{\vee}$ is identified with $2\alpha/(\alpha, \alpha)$ in $X^*(\tilde \TT)_{\QQ}$ via $(,)$, $\mu^{\ad}$ is identified with the linear combination $$\sum_{\alpha \in \Delta}\langle \alpha, \mu \rangle\frac{2}{(\alpha, \alpha)} \eta(\alpha) $$ in the basis of fundamental weights in $X^*(\tilde \TT)_{\QQ}$. Therefore, in Theorem~\ref{th-main},~\ref{item-grif-fund-weights} and~\ref{item-indep-alpha} jointly imply~\ref{item-griffiths-explicit}.
In turn,~\ref{item-griffiths-explicit} trivially implies~\ref{item-ray-equality}. Thus it suffices to prove  ~\ref{item-grif-fund-weights} and~\ref{item-indep-alpha}. Part~\ref{item-grif-fund-weights} is proved in Lemma~\ref{lem-compute-grif-alpha}. Part~\ref{item-indep-alpha} is shown first for roots with the same length in Lemma~\ref{lem-sum-same-length}. The general case is reduced to that of equal length in Lemma~\ref{lem-sum-diff-length}.

\begin{lemma}
\label{lem-compute-grif-alpha}
For all $\alpha \in \Delta$, one has
\begin{equation}
\label{eq-lemma-compute-grif-alpha}
\langle \grif(\GG,\mu,r), \alpha^{\vee} \rangle
=
-\frac{1}{2} \langle \alpha, \mu \rangle
S(\alpha^{\vee}, r).
\end{equation}
\end{lemma}
\begin{proof}
By definition of the Griffiths module~\eqref{eq-def-Grif}, the set of $\TT_{\Fbar}$-weights of $r$ is contained in that of $\Grif(\GG, \mu, r)$ (not counting multiplicities): $\Phi(\Grif(\GG, \mu, r), \TT) \subset \Phi(r, \TT)$ (\S\ref{sec-weights}).
To understand $\Grif(\GG, \mu,r)$, we must determine the multiplicity of a weight $\chi$ in $\Grif(\GG, \mu,r)$ in terms of its multiplicity in $(V,r)$. Let $\LL_V:=
\cent_{GL(V_{\CC})}(r\circ \mu)$.
Choose a maximal torus $\TT_V$ in $\LL_V$ containing $r(\TT)$.
Let $\id:GL(V) \to GL(V)$ be the identity. By construction, we have an equality of $\LL$-representations:
$$ \Grif(\GG, \mu, r)=
r^*\Grif(GL(V), r \circ \mu, \id).
$$

The $\TT_V$-weights of $\id$ all have multiplicity one ($\id$ is minuscule).
The multiplicity of a $\TT_V$-weight $\tilde \chi$ in $\Grif(GL(V), r \circ \mu, \id)$ is given by the distance to the top of the filtration, namely \begin{equation}
\label{eq-mult-GL(V)}
(r\circ \mu)_{\max}
-\langle \tilde \chi, r \circ \mu \rangle
.
\end{equation} Let $\chi=r^*\tilde \chi$. Then
$\langle \tilde \chi, r \circ \mu \rangle
=\langle\chi, \mu \rangle $.
Since there are precisely $m_V(\chi)$ different $\TT_V$-weights which pull back to $\chi$, the multiplicity of $\chi$ in $\Grif(\GG,\mu, r)$ is
\begin{equation}
\label{eq-mult-G}
m_V(\chi)( (r \circ \mu)_{\max}-\langle \chi, \mu \rangle).  \end{equation}
Since $\grif(\GG,\mu,r)$ is defined~\eqref{eq-def-grif} as the determinant of $\Grif(\GG,\mu,r)$,
\begin{equation}
\label{eq-grif-alpha-sum}
\langle \grif(\GG,\mu,r), \alpha^{\vee} \rangle
=\sum_{\chi \in \Phi(V, \TT)}
m_{\Grif(\GG,\mu,r)}(\chi)
\langle \chi, \alpha^{\vee}
\rangle.
\end{equation}
If $\chi=s_{\alpha} \chi$, then
$\langle \chi, \alpha^{\vee} \rangle=0$, so $\chi$ contributes zero in the sum~\eqref{eq-grif-alpha-sum}.
We group the remaining terms in~\eqref{eq-grif-alpha-sum} in pairs $\{\chi, s_{\alpha} \chi \}$.
Recalling that
$\langle s_{\alpha} \chi, \alpha^{\vee} \rangle=
-\langle \chi, \alpha^{\vee}\rangle $ and that the multiplicities  $m_V(\chi)$ are $W$-invariant gives
\begin{equation}
\label{eq-salpha-chi-minus-chi}
\langle \grif(\GG,\mu,r), \alpha^{\vee} \rangle=
\frac{1}{2}
\sum_{\chi \in \Phi(V, \TT)}
m_V(\chi) \langle \chi, \alpha^{\vee} \rangle \langle s_{\alpha} \chi -\chi, \mu \rangle.
\end{equation}
Substituting the definition $s_{\alpha}\chi=\chi-\langle \chi, \alpha^{\vee}  \rangle \alpha$ into~\eqref{eq-salpha-chi-minus-chi} yields~\eqref{eq-lemma-compute-grif-alpha}.
\end{proof}
Recall that $\mu$ is normalized to be $\Delta$-dominant.
\begin{corollary}
\label{cor-anti-ample}
The Griffiths character $\grif(\GG,\mu,r)$ is $\Delta$-anti-dominant. Furthermore, \begin{equation}
    \label{eq-anti-dom}
\langle \grif(\GG,\mu,r), \alpha^{\vee} \rangle<0
\end{equation} if and only if $\alpha \in \Delta \setminus \Delta_{\LL}$.
\end{corollary}
\begin{proof}
The anti-dominance is clear from~\eqref{eq-weight-pairing-sum} and~\eqref{eq-lemma-compute-grif-alpha}. Moreover, if $\alpha \in \Delta_{\LL}$, then $\langle \alpha, \mu \rangle=0$ since $\LL=\cent_{\GG_{\CC}}(\mu)$, so $\langle \grif(G,h,r), \alpha^{\vee} \rangle=0
$ again by~\eqref{eq-lemma-compute-grif-alpha}.

Finally, assume $\alpha \in \Delta \setminus \Delta_{\LL}$. Since $\mu$ is assumed $\Delta$-dominant, one has $\langle  \alpha, \mu \rangle>0$. By Lemma~\ref{lem-weights-not-all-root-hyperplane}, one of the terms in the sum~\eqref{eq-weight-pairing-sum} is nonzero, so $S(\mu, r, \alpha)<0$.
\end{proof}
\begin{lemma}
\label{lem-sum-same-length}
Assume $\GG^{\ad}$ is a simple $F$-group.
If $\alpha, \beta \in \Phi$ have the same length, then
\begin{equation}
\label{eq-sum-same-length}
S(\alpha^{\vee}, r)
=
S(\beta^{\vee}, r)
\end{equation}
\end{lemma}
\begin{proof}
Assume $\alpha, \beta \in \Phi$ have the same length. By Lemma~\ref{lem-weyl-gal-same-length}, there exists $\sigma \in W \rtimes \galf$ such that $\sigma \alpha^{\vee} =\beta^{\vee}$.   One has
\begin{equation}
S(\beta^{\vee}, r)= S(\sigma \alpha^{\vee}, r)
=
\sum_{\chi \in \Phi(V,\TT)}m_V(\chi)
\langle
\chi, \sigma \alpha^{\vee}
\rangle^2
\end{equation}
Since $r$ is a morphism of $F$-groups and $\TT$ is an $F$-torus, the set of weights $\Phi(V, \TT)$ is stable under $W \rtimes \galf$ and $m_V(\chi)=m_V(\sigma \chi)$ for all $\sigma \in W \rtimes \galf$.
Applying the orthogonality of $W$ relative $\langle, \rangle$,  together with the substitution $\sigma^{-1} \chi \mapsto \chi$ gives~\eqref{eq-sum-same-length}.
\end{proof}
We pause to note that we have shown the simply-laced case of the main result:
\begin{proof}[Proof of \textnormal{Corollary~\ref{cor-main-simply-laced}:}] Combine Lemma~\ref{lem-compute-grif-alpha}, Corollary~\ref{cor-anti-ample} and Lemma~\ref{lem-sum-same-length}.
\end{proof}
\subsubsection{Root strings}
\label{sec-root-strings}
To treat the multi-laced case, given two roots $\alpha, \beta \in \Phi$, we will use what Knapp calls the "$\alpha$-root-string containing $\beta$"
(see the paragraph preceding \Prop~2.48 in \cite{Knapp-beyond-intro-book}) and what Bourbaki call the "$\alpha$-cha\^ine de racines d\'efinie par $\beta$",
\cite[ \Chap~6, \no ~1.3]{bourbaki-lie-4-6}. The preceding synonymous terminology refers to the set $(\beta+\ZZ\alpha) \cap \Phi$.

Assume $\alpha,\beta \in \Delta$ is a pair of adjacent simple roots. Recall from \cite[\Chap~6,~\no~1.3, \Prop~9]{bourbaki-lie-4-6}) that the length $\card((\beta+\ZZ\alpha) \cap \Phi)-1$ of the $\alpha$-root string containing $\beta$ is then given as follows: If $\alpha$ is strictly longer than $\beta$, then the root string is $(\beta+\ZZ \alpha) \cap \Phi=\{\beta,\alpha+\beta\}$. Otherwise, $\alpha$ is at least as short as $\beta$ and the length of the root string is
\begin{equation}
\label{eq-compute-root-string}
\card((\beta+\ZZ\alpha) \cap \Phi)-1=\langle \alpha, \beta^{\vee} \rangle \langle \beta, \alpha^{\vee} \rangle.
\end{equation}
Note that the right-hand side of~\eqref{eq-compute-root-string} is the number of edges connecting $\alpha$ and $\beta$ in the Dynkin diagram of $\Delta$, as defined for instance in \cite[\S9.5.1]{Springer-Linear-Algebraic-Groups-book}. 
Moreover, if $\alpha$ is strictly shorter than $\beta$, then all of the members of the root string $(\beta+\ZZ\alpha) \cap \Phi$ are short, except for the two endpoints $\beta$ and $\beta+\langle \alpha, \beta^{\vee} \rangle \langle \beta, \alpha^{\vee} \rangle \alpha$ which are both long.\footnote{The last claim is checked by expanding $(\beta+k \alpha, \beta+k \alpha)/(\alpha, \alpha)$ and plugging in the value of $2(\beta, \alpha)/(\alpha, \alpha)=\langle \beta, \alpha^{\vee}\rangle$.}
\begin{lemma}
\label{lem-sum-diff-length}
Assume $\GG^{\ad}$ is a simple $F$-group and $\alpha, \beta \in \Phi$. Then
\begin{equation}
\label{eq-sum-diff-length}
\frac{S(\alpha^{\vee}, r)}{S(\beta^{\vee}, r) }
=
\frac{(\beta, \beta)}{(\alpha, \alpha)}
.
\end{equation}
\end{lemma}
\begin{proof}
By Lemma~\ref{lem-sum-same-length} and its proof, since $\GG^{\ad}$ is simple over $F$ we reduce to the case that
\begin{enumerate}
    \item $\Delta$ is multi-laced,
    \item $\alpha, \beta \in \Delta$
    \item $\alpha, \beta$ are adjacent (in particular they belong to the same $\Fbar$-simple factor of $\GG^{\ad}$), and
    \item $\alpha$ is strictly longer than $\beta$.
\end{enumerate}

 Given $\gamma, \delta \in \Phi$,
let
$$T(\gamma^{\vee}, \delta^{\vee}, r):=\sum_{\chi \in \Phi(V, \TT)}\langle \chi, \gamma^{\vee} \rangle \langle \chi, \delta^{\vee} \rangle.$$
Since $\alpha, \beta$ are adjacent, $\alpha^{\vee}+\beta^{\vee}$ is a coroot.\footnote{Caution: While $\alpha+\beta$ is a root and $\alpha^{\vee}+\beta^{\vee}$ is a coroot, one has $(\alpha+\beta)^{\vee} \neq \alpha^{\vee}+\beta^{\vee}$ whenever $\alpha, \beta$ are adjacent of different lengths.} Since $\alpha$ is strictly longer than $\beta$, $\alpha^{\vee}$ is strictly shorter than $\beta^{\vee}$.

We now apply \S\ref{sec-root-strings}, but to the root system $\Phi^{\vee}$ of coroots. Thus, $\alpha^{\vee}+\beta^{\vee}$ is a short coroot and $\beta^{\vee}+2\alpha^{\vee}$ is a coroot. The coroot $\beta^{\vee}+2\alpha^{\vee}$ is long when $\Delta$ is doubly-laced (types $B_n,C_n,F_4$) but short when $\Delta$ is triply-laced (type $G_2)$.

One has
\begin{subequations}
\begin{equation}
\label{eq-s-t-1}
    S(\alpha^{\vee}+\beta^{\vee},r)=S(\alpha^{\vee},r)+S(\beta^{\vee},r)+2T(\alpha^{\vee},\beta^{\vee},r),
\end{equation}
\begin{equation}
\label{eq-s-t-2}
S(2\alpha^{\vee}+\beta^{\vee},r)=4S(\alpha^{\vee},r)+S(\beta^{\vee},r)+4T(\alpha^{\vee},\beta^{\vee},r).
\end{equation}
\end{subequations}
On the other hand, it follows from Lemma~\ref{lem-sum-same-length} that $S(\alpha^{\vee}+\beta^{\vee},r)=S(\alpha^{\vee}, r)$, since both $\alpha^{\vee}$ and $\alpha^{\vee}+\beta^{\vee}$ are short. Also, $S(2\alpha^{\vee}+\beta^{\vee},r)=S(\beta^{\vee}, r)$ or $S(2\alpha^{\vee}+\beta^{\vee},r)=S(\alpha^{\vee}, r)$ according to whether $\Delta$ is doubly or triply-laced respectively. Solving the two equations in each of the two cases for the ratio $S(\beta^{\vee}, r)/ S(\alpha^{\vee},r)$, one finds $S(\beta^{\vee}, r)/ S(\alpha^{\vee},r)=2$ (resp. $S(\beta^{\vee}, r)/ S(\alpha^{\vee},r)=3$) when $\Delta$ is doubly (resp. triply) laced.
\end{proof}
The proof of the main result is now complete:
\begin{proof}[Proof of \textnormal{Theorem~\ref{th-main}:}] Combine Lemma~\ref{lem-compute-grif-alpha}, Corollary~\ref{cor-anti-ample} and Lemma~\ref{lem-sum-diff-length}.
\end{proof}
\subsection{Nefness of the Griffiths bundle of \texorpdfstring{$\GGZip^{\mu}$}{G-Zip}-schemes}
\label{sec-nefness}

\begin{proof}[Proof of \textnormal{Corollary~\ref{cor-nef}:}]
Let $\zeta:X \to \GGZip^{\mu}$ as in Corollary~\ref{cor-nef}.
 Write $\Gscr:=\grif(\GGZip^{\mu}, r)$ for the Griffiths line bundle  on $\GGZip^{\mu}$ associated to some choice of $r$ with central kernel (\S\ref{sec-hodge-line-bundle-GZip}) and put $\Gscr_X:=\zeta^*\Gscr$. In view of Theorem~\ref{th-main} nothing in the argument below will depend on the auxiliary choice of $r$.

Assume that $\mu$ is orbitally $p$-close (\S\ref{sec-cond-characters}). By Theorem~\ref{th-main}, the Griffiths ray is also orbitally $p$-close (by definition, one element of a ray is orbitally $p$-close if and only if all are). Thus the Griffiths character is a \underline{Hasse generator} of $\GGZip^{\mu}$ by \cite[\Th~3.2.3]{Goldring-Koskivirta-Strata-Hasse}, \ie for every stratum $\Xcal_w$ with Zariski closure $\overline{\Xcal}_w$ there exists $n_w>0$ and a section $h_w \in H^0(\overline{\Xcal}_w, \Gscr^{n_w})$ such that the nonvanishing locus of $h_w$ is precisely $\Xcal_w$.

In \cite{Brunebarbe-Goldring-Koskivirta-Stroh-ampleness}, it is shown that any line bundle corresponding to any Hasse generator is nef on a projective $\GGZip^{\mu}$-scheme $X$. Indeed, suppose $Z \subset X$ is a closed, integral subscheme. Then we have a nonzero section $h_Z$ in $H^0(Z, \Gscr_X)$ as follows: Since $Z$ is irreducible, there exists a unique stratum $\Xcal_w$ in $\GGZip^{\mu}$ such that the intersection $\zeta^*\Xcal_w \cap Z$ is dense in $Z$. Then take $h_Z:=\zeta^*h_w$.
\end{proof}
\appendix
\section{Deligne's Simplification}
\label{sec-deligne-simplification}
We present Deligne's elegant simplification of the proof of Theorem~\ref{th-main}\ref{item-ray-equality} following a letter from him \cite{Deligne-griffiths-character-letter}. In particular, Deligne's argument does not use root strings and it does not differentiate according to the lacing of the Dynkin diagram of $\Delta$.

Deligne's argument works in the setting of semisimple $\GG$, so we first explain in \S\ref{sec-app-reduction-semisimple} how to reduce Theorem~\ref{th-main}\ref{item-ray-equality} to that case. The reduction is a formality, so the reader who believes it may happily jump to Deligne's argument proper in \S\ref{sec-app-Deligne-argument}.

\subsection{Reduction to the semisimple case}
\label{sec-app-reduction-semisimple}
The proof of Lemma~\ref{lem-compute-grif-alpha}, up to~\eqref{eq-mult-G}, gives the equality of characters
\begin{equation}
\label{eq-app-grif-alpha-sum}
\grif(\GG,\mu,r)
=\sum_{\chi \in \Phi(V, \TT)}
m_V(\chi)( (r \circ \mu)_{\max}-\langle \chi, \mu \rangle)
\chi.
\end{equation}
Note that the next step~\eqref{eq-grif-alpha-sum} in the proof of Lemma~\ref{lem-compute-grif-alpha} is the result of pairing both sides of~\eqref{eq-app-grif-alpha-sum} with $\alpha^{\vee}$ but we now avoid doing this. As $\TT$-representations, one has 
\begin{equation}
\label{eq-app-constant-term}
\sum_{\chi \in \Phi(V, \TT)}
m_V(\chi)
\chi=\det V.
\end{equation}
Let $\iota:\GG^{\der} \to \GG$ be the inclusion. Since $\GG^{\der}$ is semisimple, the pullback $\iota^*\det V$ of $\det V$ to $\GG^{\der}$ is trivial. Therefore the image of ~\eqref{eq-app-constant-term} in $X^*( \TT^{\der})$ is zero. Pulling back~\eqref{eq-app-grif-alpha-sum} to $X^*(\TT^{\der})$, the first term on the right pulls back to zero, so
\begin{equation}
\label{eq-app-simple-grif-alpha-sum}
\iota^*\grif(\GG,\mu,r)
=-\sum_{\chi \in \Phi(V, \TT)}
m_V(\chi)\langle \chi, \mu \rangle
\iota^*\chi.
\end{equation}
It follows from~\eqref{eq-app-simple-grif-alpha-sum} that the validity of Theorem~\ref{th-main}\ref{item-ray-equality} is invariant under replacing $\mu$ with a (strictly) positive scalar multiple. Hence we may assume that $\mu^{\ad}$ arises from a cocharacter of $\TT_{\Fbar}^{\der}$, which we denote the same way. For such $\mu^{\ad}$, the cocharacter $\mu-\iota_*\mu^{\ad}$ of $\TT_{\Fbar}$ is central in $\GG_{\Fbar}$. 

For any cocharacter $\nu \in X_*(\TT)$ which is central in $\GG_{\Fbar}$, the sum 
\begin{equation}
\label{eq-app-sum-central-cocharacter}
\sum_{\chi \in \Phi(V, \TT)}m_V(\chi)\langle \chi, \nu \rangle \chi 
\end{equation}
is a character of $\GG_{\Fbar}$. One can see this in two ways: Since $\nu$ is central, by Schur's Lemma it acts on each $\GG_{\Fbar}$-irreducible constituent of $V_{\Fbar}$ by a character of $\GG_{\Fbar}$; the action of $\nu$ on $\det V_{\Fbar}$ is then given by~\eqref{eq-app-sum-central-cocharacter}. Alternatively, one checks that the pairing of~\eqref{eq-app-sum-central-cocharacter} with every coroot $\alpha^{\vee}$ is zero by the method of proof in Lemma~\ref{lem-compute-grif-alpha}, \ie grouping the summands with $\langle \chi, \alpha^{\vee} \rangle \neq 0$ in pairs $\{\chi, s_{\alpha} \chi \}$. As before, the pullback of~\eqref{eq-app-sum-central-cocharacter} to $X^*(\TT^{\der})$ is zero.

Applying the above with $\nu=\mu-\iota_*\mu^{\ad}$ and using~\eqref{eq-app-simple-grif-alpha-sum}, we conclude that 
\begin{equation}
\label{eq-app-grif-derived}
\iota^*\grif(\GG, \mu, r)=-\sum_{\chi \in \Phi(V, \TT)}m_V(\chi)\langle \iota^*\chi, \mu^{\ad} \rangle \iota^*\chi= \grif(\GG^{\der}, \mu^{\ad}, r^{\der}),    
\end{equation}
where $r^{\der}=\iota^*r$ is the restriction of $r$ to $\GG^{\der}$. It follows that Theorem~\ref{th-main}\ref{item-ray-equality} for $(\GG, \mu^{\ad}, r^{\der})$ implies it for $(\GG, \mu, r)$.
\subsection{Deligne's argument}
\label{sec-app-Deligne-argument}
By \S\ref{sec-app-reduction-semisimple}, we may henceforth assume that $\GG$ is semisimple. We may also assume that $\mu \neq 0$, as otherwise both sides of Theorem~\ref{th-main}\ref{item-ray-equality} are trivial.  Define a pairing on $X_*(\TT)$ by setting, for all $\mu, \mu' \in X_*(\TT)$,
\begin{equation}
\label{eq-app-perfect-pairing}
    (\mu,\mu')_V=\sum_{\chi \in \Phi(V, \TT)}m_V(\chi)\langle \chi, \mu \rangle\langle \chi, \mu' \rangle
\end{equation}
Since the $\GG$-representation $V$ is defined over $F$, the pairing $(,)_V$ is $W \rtimes \galf$-invariant. It is positive definite by Corollary~\ref{cor-weights-not-all-root-hyperplane} and thus induces an isomorphism $X^*(\TT)_{\QQ} \stackrel{\sim}{\rightarrow} X_*(\TT)_{\QQ}$. Under this isomorphism, $\grif(\GG, \mu, r)$ maps to $-\mu$ by~\eqref{eq-app-simple-grif-alpha-sum}. Since $\GG^{\ad}$ is $F$-simple, any other positive definite, $W \rtimes \galf$-invariant pairing $(,)$ on $X_*(\TT)_{\QQ}$ is a positive scalar multiple of $(,)_V$ (\S\ref{sec-root-system}). This proves Theorem~\ref{th-main}\ref{item-ray-equality}.  
\begin{rmk}
\label{rmk-app-canonical-bilinear-form}
When $V=\Ad$ is the adjoint representation, the pairing~\eqref{eq-app-perfect-pairing} is the inverse of the "canonical bilinear form" of \cite[\Chap~6,~\S1.12]{bourbaki-lie-4-6}.

\end{rmk}

\bibliographystyle{plain}
\bibliography{biblio_overleaf}

\end{document}